\documentclass[12pt]{amsart}
\usepackage{times}
\usepackage[left=1.5in, right=1.5in, top=1in, bottom=1in]{geometry}
\usepackage{amsfonts, amsmath, amsthm, amssymb, latexsym}
\usepackage{mathrsfs}
\usepackage{tikz}
\usepackage[hidelinks]{hyperref}
\usepackage{tikz-cd}
\usepackage{url}

\newtheorem{manualconjectureinner}{Conjecture}
\newenvironment{manualconjecture}[1]{%
  \renewcommand\themanualconjectureinner{#1}%
  \begin{manualconjectureinner}%
}{%
  \end{manualconjectureinner}%
}

\newtheorem{theorem}{Theorem}[section]
\newtheorem{lemma}[theorem]{Lemma}
\newtheorem{proposition}[theorem]{Proposition}
\newtheorem{corollary}[theorem]{Corollary}
\newtheorem*{conjecture*}{Conjecture}
\theoremstyle{definition}
\newtheorem{definition}[theorem]{Definition} 
\theoremstyle{remark}
\newtheorem{remark}[theorem]{Remark}
\newtheorem{example}[theorem]{Example}

\DeclareMathOperator{\ch}{char}
\DeclareMathOperator{\Sp}{Spec}
\DeclareMathOperator{\vol}{vol}
\DeclareMathOperator{\Diff}{Diff}
\DeclareMathOperator{\coeff}{coeff}
\DeclareMathOperator{\re}{red}
\DeclareMathOperator{\Supp}{Supp}

\newcommand{\pr}{{\mathbb{P}}}

\newcommand{\Q}{{\mathbb{Q}}}
\newcommand{\R}{{\mathbb{R}}}
\newcommand{\C}{{\mathbb{C}}}
\newcommand{\oo}{{\mathcal{O}}}
\newcommand{\h}{\widehat}
\newcommand{\p}{\pi}
\newcommand{\ta}{\tau}
\newcommand{\D}{\Delta}

\newcommand{\lr}{\longrightarrow}

\newcommand{\ol}{\overline}

\title{Stable Reduction via the Log Canonical Model}
\author{Tai-Hsuan Chung}

\address{Department of Mathematics, University of California San Diego, 9500 Gilman Drive \# 0112, La Jolla, CA  92093-0112, USA}
\email{t2chung@ucsd.edu}

\begin{document}

\begin{abstract}
We formulate a stable reduction conjecture that extends Deligne--Mumford's stable reduction to higher dimensions and provide a simple proof that it holds in large characteristic, assuming two standard conjectures of the Minimal Model Program. As a result, we recover the Hacon--Kovács theorem on the properness of the moduli stack $\overline{\mathscr{M}}_{2,v,k}$ of stable surfaces of volume $v$ defined over $k=\overline{k}$, provided that $\operatorname{char}k>C(v)$, a constant depending only on $v$.
\end{abstract}

\maketitle

\section{Introduction}

Moduli spaces of varieties are a central theme in algebraic geometry. One fundamental property of these moduli spaces is their properness, which is equivalent, via the valuative criterion of properness, to \emph{stable reduction} of varieties. In this note, we make some observations that offer a natural perspective on stable reduction in any dimension. In fact, our observations lead to the formulation of a stable reduction conjecture that extends the statement of Deligne--Mumford's stable reduction theorem \cite[Corollary 2.7]{DM69} to higher dimensions, with stability aligned with that introduced by Kollár--Shepherd-Barron \cite{KSB88} and Alexeev \cite{Ale96}. We formulate the \textbf{Stable Reduction Conjecture} as follows:

\phantomsection
\begin{manualconjecture}{SR\(_n\)}\label{SRn}
Let $R$ be a discrete valuation ring with fraction field $K$ and algebraically closed residue field $k$. Let $X_K$ be a projective, geometrically normal, and geometrically connected variety of dimension $n$ over $K$, and let $\D_K$ be an effective divisor on $X_K$ such that $K_{X_K}+\D_K$ is log canonical and ample over $K$. Then there exists a finite field extension $L$ of $K$ and a stable log model $(X',\D')$ over $R_L$, the integral closure of $R$ in $L$, with generic fibre $\simeq (X_K,\D_K)\times_KL$.
\end{manualconjecture}

In other words, \hyperref[SRn]{Conjecture SR\(_n\)} asserts that any such pair $(X_K,\D_K)$ has \emph{stable reduction} after a finite field extension $L/K$. This means, in our terminology, that after the extension $L/K$, there exists a log canonical model (a log canonical pair $(X,\D+X_k)$ such that $K_X+\D$ is ample) over the integral closure $R_L$ of $R$ in $L$, with central fibre $X_k$ reduced and generic fibre $(X_L,\D_L)\simeq (X_K,\D_K)\times_KL$ (c.f. Definitions \ref{slm} and \ref{asr}). 

\hyperref[SRn]{Conjecture SR\(_n\)} was established in any dimension $n$ over $k=\C$ by \cite[Corollary 1.5]{HX13}; see also \cite{Bir12}. We note that \cite[Section 7]{HX13} explains how stable reduction for demi-normal $X_K$ can be deduced from stable reduction for normal $X_K$ via Kollár's gluing theory \cite[Chapter 5]{Kol13}. This method is also illustrated in the proof of \ref{SR2}.

Our main Theorem \ref{SRn1} shows that \hyperref[SRn]{Conjecture SR\(_n\)} holds in $\ch k>C(n,v,I)$, where $C(n,v,I)$ depends only on $n=\dim X_K$, the volume $v=(K_{X_K}+\D_K)^n$, and a DCC set $I$ containing the coefficients of $\D_K$, assuming \hyperref[LCM]{Conjecture LCM\(_{n+1}\)} (existence of LCM) and \hyperref[Vn]{Conjecture V\(_n\)} (boundedness of volume) stated below. 

\phantomsection
\begin{manualconjecture}{LCM\(_{n+1}\)}\label{LCM}
With the notation in \hyperref[SRn]{Conjecture SR\(_n\)}, there exists a log canonical model $(X,(X_k)_{\re}+\D)$ over $\Sp R$ with generic fibre $\simeq (X_K,\D_K)$.
\end{manualconjecture}

Following \cite[Definition 1.5]{Pat17}, we say that a pair $(Y,B)$ is a \emph{stable log variety} if $Y$ is a projective variety over an algebraically closed field, with $K_Y+B$ semi-log canonical and ample. In the case where $Y$ is normal, $K_Y+B$ is log canonical and ample. The volume of a stable log variety $(Y,B)$ is $\vol(K_Y+B)=(K_Y+B)^{\dim Y}.$ For our purposes, we formulate the following

\phantomsection
\begin{manualconjecture}{V\(_n\)}\label{Vn}
Fix $n>0$ and a DCC set $\mathscr{C}$. Then the set \{$\vol(K_Y+B)$\}, where $(Y,B)$ is an $n$-dimensional normal stable log variety defined over an algebraically closed field $k$ with $\coeff(B)\subseteq\mathscr{C}$, has a minimal element $v(n,\mathscr{C})$ depending only on $n$ and $\mathscr{C}$, i.e. $\vol(K_Y+B)\ge v(n,\mathscr{C})>0$ for every such $(Y,B)$.
\end{manualconjecture}

Here, we say that $\mathscr{C}\subset\R$ is a \emph{DCC set} if every non-increasing sequence $(a_i)_{i\in\mathbb{N}}$ of elements in $\mathscr{C}$ stabilizes for $i\gg0$. \hyperref[Vn]{Conjecture V\(_n\)} holds in dimension $n=2$ as a special case of a more general result established in \cite[Theorem 8.2]{Ale94} and \cite[Theorem 2]{HK16}. We apply \hyperref[Vn]{Conjecture V\(_n\)} with $\mathscr{C}=D(I)$ for a DCC set $I$, where $D(I)$ is defined in Lemma \ref{MP}. We can now state our main theorem:

\begin{theorem}\label{SRn1}
Fix $n,v>0$ and a DCC set $I\subset(0,1]$ that contains $1$. Assume that \hyperref[LCM]{Conjecture LCM\(_{n+1}\)} and \hyperref[Vn]{Conjecture V\(_n\)} hold. If $\ch k>\frac{v}{v(n,D(I))}$, a constant depending only on $n$, $v$, and $I$, $\vol(K_{X_K}+\D_K)\le v$, and $\coeff(\D_K)\subseteq I$, then \hyperref[SRn]{Conjecture SR\(_n\)} holds.
\end{theorem}

The assumptions of Theorem \ref{SRn1} hold when $n=2$, thus \hyperref[SRn]{Conjecture SR\(_2\)} holds over any algebraically closed field $k$ of sufficiently large characteristic: 

\begin{corollary}[Stable reduction for lc canonically polarized surfaces]\label{SR2lc}
Fix $v>0$ and a DCC set $I\subset(0,1]$ that contains $1$. If everything is defined over an algebraically closed field $k$ of characteristic $>\operatorname{max}\{5,\frac{v}{v(2,D(I))}\}$, with $\vol(K_{X_K}+\D_K)\le v$ and $\coeff(\D_K)\subseteq I$, 
then \hyperref[SRn]{Conjecture SR\(_2\)} holds.
\end{corollary}

Together with the results of \cite{Pos24} and \cite{ABP23}, Corollary \ref{SR2lc} recovers the Hacon--Kovács Theorem \cite[Theorem 4]{HK16} on the properness of the moduli stack $\overline{\mathscr{M}}_{2,v,k}$ of stable surfaces, and, along with \cite[Theorem 1.2]{Pat17}, establishes the projectivity of the coarse moduli space $\overline{\text{M}}_{2,v,k}$ of stable surfaces, both of which hold over an algebraically closed field $k$ of sufficiently large characteristic:

\begin{theorem}[Hacon--Kovács--Patakfalvi]\label{proper} For any fixed $v_0>0$, any algebraically closed field $k$ of characteristic $>\operatorname{max}\{5,\frac{v_0}{v(2,D(\{1\}))}\}$, and any $v\le v_0$, the moduli stack $\overline{\mathscr{M}}_{2,v,k}$ of stable surfaces of volume $v$ over $k$ is proper, and  the coarse moduli space $\overline{\emph{M}}_{2,v,k}$ of stable surfaces of volume $v$ over $k$ is projective.
\end{theorem}

For precise definitions of $\overline{\mathscr{M}}_{2,v,k}$ and $\overline{\text{M}}_{2,v,k}$, we refer to \cite[Section 1.3]{Pat17} and \cite[Section 6]{Pos24}. 

A key observation that led to our formulation of stable reduction, and which seems to make the formulation natural, is the following result: it characterizes canonically polarized varieties that have stable reduction after a tame base change. In particular, it provides a generalization of T. Saito's characterization for curves (\cite[Theorem 3]{Sai87}, \cite[10.4.47]{Liu02}, \cite[1.2]{Hal09}) to arbitrary dimensions:

\begin{theorem}[c.f. Theorem \ref{tame}]\label{tam}
Let $Y$ be a projective, geometrically normal, and geometrically connected variety over the fraction field $K$ of a Henselian DVR $R$ with algebraically closed residue field $k$ of characteristic $p>0$, such that $K_Y$ is log canonical and ample. The following are equivalent: 

$(1)$ $Y$ has stable reduction after a tamely ramified extension $L/K$.

$(2)$ There exists a log canonical model $(X,(X_k)_{\re})$ over $\Sp R$ with generic fibre $\simeq Y$ such that every component of $X_k$ has multiplicity prime to $p$.
\end{theorem}

We will prove a log version Theorem \ref{tame} of Theorem \ref{tam}. Recall that for families of curves, a component of the central fibre is called \emph{principal} if it is either a curve of genus at least one or a rational curve with at least three marked points. Theorem \ref{tam} gives a generalization of the notion of principal components to higher dimensions: those components that stay in the log canonical model $(X,(X_k)_{\re})$ should be considered principal. When there is a principal component with multiplicity divisible by $p$, Theorem \ref{tam} implies that a base change of degree divisible by $p$ is necessary to achieve stable reduction (Corollary \ref{degp}). However, we do not know what is the right base change to make. We discuss an example (\ref{Artin}) regarding this.

\begin{remark}
Our assumptions on the characteristic is only for the residue field $k$. Thus the results hold in both mixed and equi-characteristic.
\end{remark}

\begin{remark}
In contrast to the stable reduction discussed above, \emph{semi-stable reduction} in the sense of \cite{KKMS73} asserts that given a family $f\colon X\lr T\ni0$ of varieties over a regular curve $T$ such that $f$ is smooth over $T-0$, there exists a finite base change $g\colon T'\lr T$ with $g^{-1}(0)=0'$, a regular variety $X'$, and a projective birational morphism $X'\lr X\times_TT'$ such that the induced morphism $f'\colon X'\lr T'$ has central fibre $f'^*(0')$ reduced with simple normal crossing singularities. Therefore, stable reduction in our sense is much weaker than semi-stable reduction. We do not prove semi-stable reduction in this note. 
\end{remark}

\emph{Previous results.} Stable reduction for curves was first formulated and proved in \cite[Corollary 2.7]{DM69} in any characteristic, while it remains conjectural in positive characteristic for higher dimensions. When the residue characteristic $\ch k$ is sufficiently large, stable reduction has been established for semi-log canonical canonically polarized surfaces by \cite[Theorem 4]{HK16}, and for smooth canonically polarized surfaces by \cite[Theorem 10.6]{BMP23}. The proof of the former result uses ultraproducts to reduce the problem to semi-stable reduction in characteristic $0$, while neither the bound on $\ch k$ nor the base change used is explicit in their proof. The proof of the latter result has not been generalized to the non-smooth case. Our proof of Theorem \ref{SRn1} explicitly describes both how the bound on $\ch k$ depends on $v$ and the base change used. For example, in the boundary-free case for surfaces ($\D_K=0$ and $n=2$ in \hyperref[SRn]{Conjecture SR\(_n\)}), our bound $\operatorname{max}\{5,\frac{v}{v(2,D(\{1\}))}\}$ applies to every such surface with volume at most $v$.

We now discuss the ideas behind Theorem \ref{SRn1}. The stable reduction of Deligne--Mumford \cite[Corollary 2.7]{DM69} can be interpreted as the process of constructing a family $X\lr T$ of curves, with generic fibre the given curve, such that the central fibre $X_k$ is reduced, the pair $(X,X_k)$ is log canonical, and $K_X$ is ample. Therefore, singularities of the Minimal Model Program naturally arise in the context of stable reduction, suggesting a study of the singularities of the total space $X$ together with its central fibre $X_k$. Thus we study the pair $(X,(X_k)_{\re})$ and compare its singularities under base change. For this, the key input is Proposition \ref{kol}, which effectively controls the singularities when the base change has only tame ramifications. With this in mind, we observe that by passing to a model $X$ where $K_X+(X_k)_{\re}$ is ample over the base $T$---namely, the log canonical model $(X,(X_k)_{\re})$ over $T$---and by making a tame base change $T'\lr T$ that reduces all the multiplicities of the components of $X_k$ to one, we achieve our goal: the normalization $X'$ of $X\times_TT'$ will have $X'_k$ reduced, $(X',X'_k)$ log canonical, and $K_{X'}$ ample over $T'$. The base change needed will be tame if every component of $X_k$ has multiplicity prime to the residue characteristic $p$, which is a strong restriction. Nevertheless, this situation can be achieved when $p$ is large compared to the volume $v$ of the generic fibre, by using a boundedness result. In fact, writing the central fibre as $X_k=\sum_im_iF_i$, where each $F_i$ is a component of $X_k$ with multiplicity $m_i$, we have
\begin{equation}\label{v}
v=\vol(X_k)=\sum_im_i\vol(K_{F_i}+B_i)    
\end{equation}
where $B_i$ is the intersection of $F_i$ with the rest of the central fibre. By the boundedness theorem \cite[Theorem 8.2]{Ale94} and \cite[Theorem 2]{HK16}, there is a fixed number $v_{\text{min}}>0$ such that $\vol(K_{F_i}+B_i)\ge v_{\text{min}}$ for any $i$, combining with (\ref{v}) gives
$$
\frac{v}{v_{\text{min}}}\ge\sum_im_i\ge m_i.
$$
Therefore, if $\ch k=p>\frac{v}{v_{\text{min}}}$, all the multiplicities $m_i$ are less than $p$, the base change needed will be tame, and hence the approach above works. 

The above approach provides a new perspective on stable limits. Recall that \cite{KSB88} compactifies the moduli space of varieties of general type as follows: Start with a smooth family $X_0\longrightarrow T_0$ of varieties over an affine curve $T_0$. First apply a finite base change $T'_0\longrightarrow T_0$ so that the family extends to a semi-stable family $X'\longrightarrow T'$ over a proper curve $T'\supset T'_0$ by \cite{KKMS73}; then pass to the relative canonical model over $T'$. Our approach above suggests that we can achieve the same result by reversing the steps: first pass to the log canonical model, and then make a base change $t\lr t^N$. This yields the stable model in characteristic $0$, and it also works in large characteristic by using the boundedness theorem of \cite{Ale94} and \cite{HK16}. In particular, once we have the log canonical model, we obtain a clearer picture of what the stable limit looks like (Corollary \ref{SL}). The same circle of ideas, together with the main theorem of \cite{BCHM10}, shows the existence of the log canonical model $(X,(X_k)_{\re}+\D)$ over $\C$, assuming $(X,\D)$ is klt away from the central fibre $X_k$ (Corollary \ref{lcm0}). This provides a proof of \hyperref[SRn]{Conjecture SR\(_n\)} over $k=\C$ for the case where $(X_K,\D_K)$ is klt (Corollary \ref{SRn0}). 

We should mention that Quentin Posva kindly contacted us and mentioned that he has a new approach to prove semi-stable reduction for curves and has ideas for the surface case. Our approaches seem to be different.

\emph{Structure of the note.} In section $2$, we fix the terminology and introduce our definition of stable reduction. In section $3$, we prove the results. Some of the lemmas used in the proofs are collected in the Appendix.

\textbf{Acknowledgements:} I would like to express my deep gratitude to my advisor, Professor James M\textsuperscript{c}Kernan, for suggesting the problem, numerous inspiring discussions, careful reading of several versions of this note, his feedback,
and his constant support and encouragement. I am especially grateful to Professor M\textsuperscript{c}Kernan for sharing a strategy for proving semi-stable reduction for surfaces, which inspired this work. I am deeply indebted to Professor Michael McQuillan for pointing out serious mistakes in my earlier ideas, for many valuable discussions, for sharing his preprint, and for his encouragement. I am very grateful to Jacob Keller for numerous helpful discussions, valuable comments, and encouragement. I would like to thank Professors Rusiru Gambheera, Joaquín Moraga, Qingyuan Xue, and Ziquan Zhuang for answering my questions and for their encouragement, as well as Professors Chen-Yu Chi, Bochao Kong, David Stapleton, Joe Waldron, Chin-Lung Wang, and Chenyang Xu for helpful conversations and encouragement. I thank Rahul Ajit, Fernando Figueroa, and Chung-Ming Pan for encouragement. Finally, I thank my family for their support, especially my wife, Wei-Na. I was partially supported by DMS-1802460 and a grant from the Simons Foundation.

\section{Terminology}

Here we fix the terminology and notation. We primarily follow the terminology from \cite{KM98}, \cite{Kol13}, \cite{Kol23b}, and \cite{ABP23}, with slight modifications to better suit our purposes.

Throughout the note, following \cite{ABP23}, unless otherwise stated, we work over a fixed base ring, which is assumed to be Noetherian, excellent, of finite Krull dimension, admitting a dualizing complex, and of pure dimension. Furthermore, $X$, $Y$, and $Z$ are always assumed to be quasi-projective and equidimensional schemes over this base ring. These conditions ensure that the dualizing sheaf $\omega_X$ exists, and that each pair $(X,D)$ (defined below) has a canonical divisor $K_X$, with $X$ being regular at all the generic points of $\Supp K_X$, such that $\oo_X(K_X)\simeq\omega_X$.

A \emph{variety} is a reduced connected equidimensional scheme that is separated and of finite type over a field (which is not necessarily algebraically closed). Thus for us a variety could be reducible. A curve (resp. surface, resp. threefold) is a variety of dimension one (resp. two, resp. three). 

A reduced scheme is \emph{demi-normal} if it satisfies Serre's condition $S_2$ and its codimension 1 points are either regular points or nodes. A Weil $\Q$-divisor $D$ on $X$ is called a \emph{Mumford} $\Q$-\emph{divisor} if $X$ is regular at all the generic points of $\Supp D$. A \emph{pair} $(X,D)$ consists of a demi-normal (quasi-projective and equidimensional) scheme $X$ and an effective Mumford $\Q$-divisor $D$ such that $K_X+D$ is $\Q$-Cartier. If $(X,D)$ is a pair with $X$ normal, for a proper birational morphism $\phi\colon Z\lr X$ of normal schemes and a $\phi$-exceptional divisor $E$, the \emph{discrepancy} $a(E,X,D)$ \emph{of} $E$ \emph{with respect to} $(X,D)$ is defined as $\coeff_E(K_Z-\phi^*(K_X+D))$. For non-exceptional divisors $E\subset X$, set $a(E,X,D):=-\coeff_E D$. We say that a pair $(X,D)$, or $K_X+D$, with $X$ normal, is \emph{log canonical} or \emph{lc} (resp. \emph{klt}) if $a(E,X,D)\ge-1$ (resp. $a(E,X,D)>-1$) for every divisor $E$ over $X$. Note that $a(E,X,D)$ can be defined more generally, including the case where $(X,D)$ is a pair as defined above (cf. \cite[2.4]{Kol13}). A pair $(Y,B)$, or $K_Y+B$, is \emph{semi-log canonical (slc)} if $\nu^*(K_Y+B)=K_{\widetilde{Y}}+\widetilde{D}+\widetilde{B}$ is log canonical, where $\widetilde{Y}\overset{\nu}{\lr}Y$ is the normalization, $\widetilde{D}$ is the conductor on $\widetilde{Y}$, and $\widetilde{B}$ is the divisorial part of $\nu^{-1}(B)$. 

Let $T$ be the spectrum of a field or a Dedekind domain. A \emph{pair} $(X,D)$ \emph{over} $T$ consists of a pair $(X,D)$ and a proper flat morphism $X\lr T$. We say that a pair $(X,D)$ over $T$ is a \emph{semi-log canonical model} over $T$ if $K_X+D$ is slc and ample over $T$. If, in addition, $X$ is normal, then $(X,D)$ is a \emph{log canonical model} (\emph{LCM}) over $T$. 

For a flat morphism $X\lr T$ of finite type, we denote by $X_t$ the fibre over $t\in T$. When $T$ is the spectrum of a discrete valuation ring (DVR) $R$ with fraction field $K$ and residue field $k$, we denote by $X_K$ the generic fibre and by $X_k$ the central fibre. By a \emph{component} of $X_k$ we always mean an irreducible component. We say that a $\Q$-divisor $\D$ on $X$ \emph{satisfies condition} $(\ast)$ if for every $t\in T$, the support $\Supp\D$ does not contain any irreducible component of $X_t$, and none of the irreducible components of $X_t\cap\Supp\D$ is contained in the singular locus of $X_t$. Thus the restriction of $\D$ to $X_t$ makes sense. If $K$ is the function field of $T$, the restriction of $\D$ to the generic fibre is denoted by $\D_K$.

Let $T$ be the spectrum of a Dedekind domain. A \emph{family of varieties} $f\colon X\lr T$ is a proper flat morphism such that for every $t\in T$, the fibre $X_t$ is equidimensional, geometrically reduced, and geometrically connected. A \emph{family of pairs} $f\colon(X,\D)\lr T$ is a family of varieties $f\colon X\lr T$ together with an effective Mumford $\Q$-divisor $\D$ on $X$ satisfying condition $(\ast)$ defined above, such that $(X,\D+X_t)$ is a pair for every closed point $t\in T$. 

\vspace{3mm}

We introduce our definition of stable reduction in the following:

\begin{definition}\label{slm} Let $R$ be a Dedekind domain with closed points $r_1,\cdots,r_s$. A family of pairs $f\colon(X,\D)\lr\Sp R$, with geometric generic fibre demi-normal, is a \emph{stable log model} if  $(X,\D+X_{r_1}+\cdots+X_{r_s})$ is a semi-log canonical model over $\Sp R$. Equivalently, $(X,\D+X_{r_i})$ is slc for every $r_i$, and $K_X+\D$ is ample over $\Sp R$. 
\end{definition}

\begin{definition}\label{asr}
Let $Y$ be a projective, geometrically demi-normal, and geometrically connected variety over the fraction field $K$ of a DVR $R$ with algebraically closed residue field, and let $\D_Y$ be an effective divisor on $Y$ such that $K_Y+\D_Y$ is slc and ample over $K$. We say that $(Y,\D_Y)$ \emph{has stable reduction} after a finite field extension $L/K$ if there exists a stable log model $(X,\D)$ over $\Sp R_L$ with generic fibre $(X_L,\D_L)\simeq(Y,\D_Y)\times_KL$, where $R_L$ is the integral closure of $R$ in $L$.
\end{definition}

\begin{remark}
When $Y$ is smooth of dimension $1$ and $\D_Y=0$, Definition \ref{slm} aligns with \cite[1.1]{DM69}, and Definition \ref{asr} aligns with \cite[2.2]{DM69}. In general, Definition \ref{slm} corresponds to the notion of stability introduced by \cite{KSB88} and \cite{Ale96}, and is equivalent to the definition of a \emph{stable family} given in \cite[4.11]{ABP23}. 
\end{remark}

\begin{remark}
In both Definitions \ref{slm} and \ref{asr}, the assumption of geometric demi-normality is necessary for the notion of stable log model to be \emph{stable} after any finite field extensions, as the generic fibre of a stable log model is demi-normal. 
\end{remark}

\section{Proofs}

We begin with a characteristic-free characterization of log canonically polarized varieties that have stable reduction after a tamely ramified base change:

\begin{theorem}\label{tame}
Let $Y$ be a projective, geometrically normal, and geometrically connected variety over the fraction field $K$ of a Henselian DVR $R$ with algebraically closed residue field $k$, and let $\D_Y$ be an effective divisor on $Y$ such that $K_Y+\D_Y$ is log canonical and ample over $K$. Fix an integer $N>0$ not divisible by $\ch k$. The following are equivalent:

$(1)$ $(Y,\D_Y)$ has stable reduction after a field extension $L/K$ of degree $N$.

$(2)$ There exists a log canonical model $(X,\D+(X_k)_{\re})$ over $\Sp R$ with generic fibre $\simeq(Y,\D_Y)$ such that the least common multiple $l$ of the multiplicities of all the components of $X_k$ divides $N$.

In particular, $l$ is the minimal degree of a tame base change for which $(Y,\D_Y)$ has stable reduction. 
\end{theorem}

A key ingredient in the proof of Theorem \ref{tame} is the following, where we refer to \cite[2.39]{Kol13} for the definition of \emph{ramified covers}: 

\begin{proposition}\cite[2.42-2.43]{Kol13}\label{kol} 
Let $g\colon X'\longrightarrow X$ be a ramified cover between demi-normal schemes. Let $\Delta$ and $\Delta'$ be $\Q$-divisors on $X$ and $X'$, respectively, such that there is a canonical $\Q$-linear equivalence
\[
K_{X'}+\Delta'\sim_{\Q}g^*(K_X+\Delta).
\]
Then, if $K_X+\Delta$ is lc (resp. klt), so is $K_{X'}+\Delta'$. The converse holds provided that $g$ is Galois and $\ch k\nmid\deg g$. 
\end{proposition}

\begin{proof}[Proof of Theorem \ref{tame}]
$(2)\implies(1)$ Let $L$ be the field 
$$
\frac{K[t]}{\p-t^N},
$$
where $\p$ is a uniformizer of $R$. Let $R_L$ be the integral closure of $R$ in $L$, $X'$ be the normalization of $X\times_RR_L$, and $g\colon X'\longrightarrow X$ be the induced morphism. We will see that $(X',g^*\D)$ is a stable log model with generic fibre $\simeq(Y,\D_Y)\times_KL$. Since $N$ is not divisible by $\ch k$ and is a multiple of $l$, every component of $X'_k$ is reduced, ie. $(X'_k)_{\re}=X'_k$ (Lemma \ref{bcq}). Therefore, as $g\colon X'\longrightarrow X$ is tamely ramified, by Lemma \ref{bc} we obtain
\[
K_{X'}+g^*\D+X'_k\sim_{\Q}g^*(K_X+\D+(X_k)_{\re}).
\]
As $K_X+\D+(X_k)_{\re}$ is lc, so is $K_{X'}+g^*\D+X'_k$ by Proposition \ref{kol}. Since $K_X+\D+(X_k)_{\re}$ is ample over $\Sp R$, $K_{X'}+g^*\D+X'_k$ is ample over $\Sp R_L$. Thus $(X',g^*\D)$ is the stable log model.

\vspace{1mm}

$(1)\implies(2)$ Let $(X',\D')$ be a stable log model with generic fibre $(X'_L,\D'_L)\simeq(Y,\D_Y)\times_KL$. Since $R$ is Henselian with residue field $k=\ol k$ and $\ch k\nmid N$, $L/K$ is Galois with cyclic group $G=\langle\sigma\rangle$ of order $N$ (Lemma \ref{cyclic}). A generator $\sigma$ of $G$ induces an automorphism Id $\times\sigma$ of $Y\times_KL\simeq X'_L$, so $G$ acts on $X'_L$. 

We claim that $G$ acts on $X'$, i.e. the automorphism Id $\times\sigma$ on $X'_L$ extends to an automorphism on $X'$. For this we base heavily on the proof of \cite[11.40]{Kol23b}: Consider the normalization $\Gamma$ of the closure of the graph of Id $\times\sigma$ in $X'\times X'$, and let $p_1$ and $p_2$ be the projections from $\Gamma$ onto each $X'$. Suppose $E$ is a $p_1$-exceptional prime divisor. Denoting by ${p_{1_*}^{-1}}(\cdot)$ the strict transform, near the generic point of $E$ we have
$$
K_{\Gamma}+{p_{1_*}^{-1}}(\D'+X'_k)\sim_{\Q} p_1^*(K_{X'}+\D'+X'_k)+aE
$$
where $a\ge-1$, as $K_{X'}+\D'+X'_k$ is lc. Since $p_1$ is an isomorphism away from $X'_k$, the pullback $p_1^*X'_k$ (which makes sense as $X'_k$ is Cartier) contains at least one copy of $E$, so
\[
\begin{aligned}
    K_{\Gamma}+{p_{1_*}^{-1}}(\D'+X'_k)&\sim_{\Q} p_1^*(K_{X'}+\D')+p_1^*X'_k+aE\\
&\ge p_1^*(K_{X'}+\D')+{p_{1_*}^{-1}}X'_k+(1+a)E
\end{aligned}
\]
with $1+a\ge0$. Therefore, running through all $p_1$-exceptional prime divisors, we have 
\begin{equation}\label{E1}
    K_{\Gamma}+{p_{1_*}^{-1}}\D'\sim_{\Q} p_1^*(K_{X'}+\D')+E_1,
\end{equation}
where $E_1$ is an effective $p_1$-exceptional divisor. Likewise, we have
\begin{equation}\label{E2}
    K_{\Gamma}+{p_{2_*}^{-1}}\D'\sim_{\Q} p_2^*(K_{X'}+\D')+E_2
\end{equation}
where $E_2$ is effective $p_2$-exceptional. Since ${p_{1_*}^{-1}}\D'={p_{2_*}^{-1}}\D'$ (as $p_1^{-1}(\D'_L)$ and $p_2^{-1}(\D'_L)$ agree on $\Gamma$), subtracting (\ref{E2}) from (\ref{E1}) we obtain
\begin{equation}\label{3}
-(E_1-E_2)\sim_{\Q}p_1^*(K_{X'}+\D')-p_2^*(K_{X'}+\D'),    
\end{equation}
which is $p_2$-nef as $K_{X'}+\D'$ is ample. Thus by the negativity of contraction \cite[1.17]{Kol13}, $p_{2_*}(E_1-E_2)=p_{2_*}E_1\ge0$ implies $E_1-E_2\ge0$. By applying $p_{1_*}$ to (\ref{3}) we get $E_2-E_1\ge0$. Hence $E_1=E_2$ and
\[
p_1^*(K_{X'}+\D')\sim_{\Q}p_2^*(K_{X'}+\D').
\]
Suppose $p_1$ contracts a curve $C$. Then $C$ cannot be contracted by $p_2$, by the construction of $\Gamma$. But then
\[
0=p_1^*(K_{X'}+\D')\cdot C=p_2^*(K_{X'}+\D')\cdot C=(K_{X'}+\D')\cdot p_{2_*}C>0,
\]
a contradiction. Thus $p_1$ is finite, and hence an isomorphism (as $X'$ is normal). So the action of $G$ extends to the whole of $X'$, as claimed.

As $G$ acts on $X'$, we can consider the geometric quotient 
\[
g\colon X'\longrightarrow X'/G=:X.
\]
Then $X$ is defined over $(R_L)^G=R$ with generic fibre $(Y\times_KL)^G=Y$. We will show that $(X,\frac{g_*\D'}{\deg g}+(X_k)_{\re})$ is the log canonical model over $\Sp R$. Since $K_{X'}+\D'+X'_k$ is $\Q$-Cartier, so is its norm (c.f. \cite[2.40.1]{Kol13})
\begin{equation}\label{g_*}
g_*(K_{X'}+\D'+X'_k)=(\deg g)(K_X+\D+F),
\end{equation}
where $\Supp F\subseteq\Supp X_k$ and $\D$ is the rest so that none of the irreducible components of $\D$ is contained in $\Supp X_k$. Thus $g_*\D'=(\deg g)\D$, i.e. $\D=\frac{g_*\D'}{\deg g}$. Further, as $\Supp g_*X'_k=\Supp X_k$, $\Supp F=\Supp X_k$ by construction. We shall see that $F=(X_k)_{\re}$. Indeed, from (\ref{g_*}) we have $$g_*(K_{X'}+\D'+X'_k)=(\deg g)(K_X+\D+F)=g_*g^*(K_X+\D+F).$$ Thus $g_*((K_{X'}+\D'+X'_k)-g^*(K_X+\D+F))=0$, and, since $g$ is finite,
\begin{equation}\label{g^*}
K_{X'}+\D'+X'_k=g^*(K_X+\D+F).
\end{equation}
Applying Lemma \ref{bc} to (\ref{g^*}), as $\Supp F$ contains the branch divisor of $g$, $g^*\D=\D'$, and $(g^*F)_{\re}=X'_k$, $F$ is reduced, i.e. $F=(X_k)_{\re}$. Moreover, as $K_{X'}+\D'+X'_k$ is lc and $g$ is Galois with $\ch k\nmid\deg g$, $K_X+\D+F$ is lc by Proposition \ref{kol}. Now, restricting (\ref{g^*}) to the generic fibres gives $K_{X'_L}+\D'_L=g^*(K_{X_K}+\D_K)$, and hence we have 
\[
(X_K,\D_K)=(X'_L,\D'_L)^G\simeq((Y,\D_Y)\times_KL)^G=(Y,\D_Y).
\]
Thus $(X,\frac{g_*\D'}{\deg g}+(X_k)_{\re})$ is lc, with generic fibre $(X_K,\D_K)\simeq(Y,\D_Y)$.

Next, to show that $K_X+\frac{g_*\D'}{\deg g}+(X_k)_{\re}=:B$ is ample over $\Sp R$, it is enough to show that given any coherent sheaf $\mathcal{F}$ on $X$, $H^i(X,\mathcal{F}\otimes\oo_X(mB))=0$ for all $i>0$ and all sufficiently large $m$. We have the following diagram, with $g=g_L\circ\nu$:
\[
\begin{tikzcd}
X'\arrow[r,"\nu"]\arrow[rd]&X\times_RR_L\arrow[d,]\arrow[r,"g_L"] & X\arrow[d]\\
 &\Sp R_L\arrow[r] & \Sp R
\end{tikzcd}
\]
Since $\nu$ is finite surjective and $\nu^*g_L^*B=K_{X'}+\D'+X'_k$ is ample over $\Sp R_L$, $g_L^*B$ is ample over $\Sp R_L$. By flat base change \cite[III.9.3]{Har77}
\[
\begin{aligned}
H^i(X,\mathcal{F}\otimes\oo_X(mB))\otimes_RR_L
&\simeq H^i(X\times_RR_L,g_L^*(\mathcal{F}\otimes\oo_X(mB)))\\
&\simeq H^i(X\times_RR_L,g_L^*\mathcal{F}\otimes\oo_{X\times_RR_L}(mg_L^*B)).
\end{aligned}
\]
This vanishes for all $i>0$ and all sufficiently large $m$, since $g_L^*B$ is ample over $\Sp R_L$. As $R$ is local and $R\lr R_L$ is finite, $H^i(X,\mathcal{F}\otimes\oo_X(mB))\otimes_RR_L=0$ implies $H^i(X,\mathcal{F}\otimes\oo_X(mB))=0$ by Nakayama's lemma. Thus $K_X+\frac{g_*\D'}{\deg g}+(X_k)_{\re}$ is ample over $\Sp R$.

Finally, if $D$ is a component of $X_k$ with multiplicity $m_D$, and $D'$ is a component of $g^{-1}(D)$ (which is reduced), then by Lemma \ref{val} we have
\[
m_D\cdot e_{D'}=\deg g\cdot m_{D'}=\deg g.
\]
Thus $m_D$, and hence the least common multiple $l$, divides $\deg g=N$. Hence $(X,\frac{g_*\D'}{\deg g}+(X_k)_{\re})$ is the log canonical model. 
\end{proof}

\begin{corollary}\label{degp} Let $(Y,\D_Y)$ be as in Theorem \ref{tame}, and assume that it has stable reduction after a finite extension $L/K$. If there exists an LCM $(X,(X_k)_{\re}+\D)$ over $\Sp R$ with generic fibre $\simeq(Y,\D_Y)$ such that $X_k$ has a component whose multiplicity is divisible by $\ch k=p$, then the degree of $L/K$ is divisible by $p$.
\end{corollary}
\begin{proof}
Suppose to the contrary that the degree of $L/K$ is prime to $p$. Then Theorem \ref{tame} produces an LCM $(X,(X_k)_{\re}+\D)$ with generic fibre $\simeq(Y,\D_Y)$ such that every component of $X_k$ is prime to $p$. But the LCM is unique, a contradiction.
\end{proof}

\begin{remark} 
It seems interesting, though unclear to us how, to establish an analogue of Theorem \ref{tame} for cases where $K_Y+\D_Y$ is anti-ample or trivial. Once we fix a model $(X,(X_k)_{\re}+\D)$ over $\Sp R$ with generic fibre $\simeq(Y,\D_Y)$, the implication $(2)\Rightarrow(1)$ of Theorem \ref{tame} holds regardless of the positivity of $K_X+\D+(X_k)_{\re}$. However, without the ampleness of $K_Y+\D_Y$, our proof of $(1)\Rightarrow(2)$ fails in general, as the action of $G=\operatorname{Gal}(L/K)$ on the generic fibre may not extend to the total space.
\end{remark}

When there is a component with multiplicity divisible by $p$ in $X_k$, we do not know which base change would realize the stable reduction. One difficulty arises from the fact that $\p=t^p$ becomes an inseparable extension in characteristic $p$, causing wild ramification, which renders, for instance, Proposition \ref{kol} no longer useful. Another difficulty arises from the fact that there are many separable base changes of degree $p$ in characteristic $p$, and it is unclear to us which one would work. We discuss the following example, pointed out to us by Professor Michael McQuillan.

\begin{example}[\cite{Art75} and \cite{McQ}]\label{Artin} Let $C$ be a smooth curve of genus $\ge2$ over a field of characteristic two with an involution $\sigma$ having at least one fixed point such that $C/\sigma$ is rational. Consider the quotient $(C\times C)/\sigma$ by the diagonal action and the diagram
\[
\begin{tikzcd}
C\times C\arrow[r]\arrow[d] & (C\times C)/\sigma\arrow[d]\\
C\arrow[r] & C/\sigma
\end{tikzcd}
\]
Then $X':=C\times C$ is the stable model and the stable limit is a copy of $C$. The central fibre of $X:=(C\times C)/\sigma$ (with $\sigma$ satisfying extra properties if necessary) consists of a rational curve $D$ of even multiplicity $m_D$ with a single cusp, whose preimage under the base change $C\lr C/\sigma$ is the stable limit. Now, in the minimal log resolution $\h X$ of $(X,(X_k)_{\re})$, $D$ meets the rest of the central fibre at a single point and therefore gets contracted in the LCM $X^{\text{LCM}}$ of $(X,(X_k)_{\re})$ over $C/\sigma$. Thus the normalization of $X^{\text{LCM}}\times_{C/\sigma}C$ cannot be the stable model $X'$. On the other hand, $\h X$ introduces a component $E$ with multiplicity divisible by the residue characteristic $2$, which stays in $X^{\text{LCM}}$. Indeed, since $D$ meets the rest of $\h X_k$ in only one component, call it $D_1$, $m_D$ must divide $m_{D_1}$. Thus $m_{D_1}$ is even. If $D_1$ meets $\ge3$ components in $\h X_k$, then take $E=D_1$. If not, $D_1$ meets exactly one other component $D_2$ other than $D$. Then $m_{D_2}$ must be even. Continuing in this manner, we can take $E=D_i$ for some $i$. Otherwise, we would obtain a cycle of $\pr^1$s, which has arithmetic genus one and thus cannot occur. Finally, we note that the preimage of any such $E$ under the base change $C\lr C/\sigma$ is not the stable limit.
\end{example}

We now prove our main Theorem \ref{SRn1}:

\begin{proof}[Proof of Theorem \ref{SRn1}]
Start with the log canonical model $(X,\D+(X_k)_{\re})$ over $\Sp R$ with generic fibre $\simeq (X_K,\D_K)$. Write the central fibre $X_k=\sum_im_iF_i$, where $F_i$ is a component of $X_k$ with multiplicity $m_i$. Since $X\lr\Sp R$ is flat and $K_X+\D+(X_k)_{\re}$ is ample over $\Sp R$, we have
\begin{equation}\label{vge}
\begin{aligned}
   v&\ge\vol(K_{X_K}+\D_K)\\
    &=(K_X+\D+(X_k)_{\re}|_{X_K})^n\\
    &=(K_X+\D+(X_k)_{\re}|_{X_k})^n\\
    &=(K_X+\D+(X_k)_{\re})^n\cdot(\sum_im_iF_i)
    \\
    &=\sum_im_i(K_X+\D+(X_k)_{\re})^n\cdot F_i\\
\end{aligned}
\end{equation}
Here we note that, in contrast to the case of $\ch k=0$, where $(K_X+\D+(X_k)_{\re})|_{F_i}=K_{F_i}+\Diff(\D+(X_k)_{\re}-F_i)$ will be slc, in $\ch k>0$, these $F_i$ may even fail to satisfy $S_2$, as explained in \cite[Remark 2]{Kol23a}. Nevertheless, this is not an issue for our purposes, since we are concerned with the intersection number $(K_X+\D+(X_k)_{\re})^n\cdot F_i$, which can be computed by passing to the normalization $F^{\nu}_i$ of $F_i$, using the projection formula for intersection products. To study the term $(K_X+\D+(X_k)_{\re})^n\cdot F^{\nu}_i$, by \cite[4.2.9]{Kol13} we have
$$(K_X+\D+(X_k)_{\re})|_{F^{\nu}_i}\sim_{\Q}K_{F^{\nu}_i}+\Diff_{F^{\nu}_i}(\D+(X_k)_{\re}-F_i).$$ We see that $(F^{\nu}_i,\Diff_{F^{\nu}_i}(\D+(X_k)_{\re}-F_i))$ is a normal stable log variety. Indeed, since $K_X+\D+(X_k)_{\re}$ is lc and ample, so is $K_{F^{\nu}_i}+\Diff_{F^{\nu}_i}(\D+(X_k)_{\re}-F_i)$, by \cite[Lemma 4.8]{Kol13}. Moreover, $\coeff(\D)=\coeff(\D_K)\subseteq I$ implies that $\coeff(\D+(X_k)_{\re}-F_i)\subseteq I$, which in turn implies 
\[
\begin{aligned}
    \coeff(\Diff_{F_i^{\nu}}(\D+(X_k)_{\re}-F_i))
    &=\coeff(\Diff_{F_i}(\D+(X_k)_{\re}-F_i))\\
    &\subseteq D(I)
\end{aligned}
\]
by \cite[4.5 (6)]{Kol13} and \cite[16.7]{F&A}, with $D(I)$ defined in Lemma \ref{MP}. Therefore, $(F^{\nu}_i,\Diff_{F^{\nu}_i}(\D+(X_k)_{\re}-F_i))$ is a normal stable log variety with $\coeff(\Diff_{F^{\nu}_i}(\D+(X_k)_{\re}-F_i))\subseteq D(I)$, which is a DCC set since $I$ is (Lemma \ref{MP}). 

Now apply \hyperref[Vn]{Conjecture V\(_n\)} with $\mathscr{C}=D(I)$. We find a constant $v(n,D(I))>0$ such that $\vol(K_{F^{\nu}_i}+\Diff_{F^{\nu}_i}(\D+(X_k)_{\re}-F_i))\ge v(n,D(I))$ for every $i$. Therefore,
\[
\begin{aligned}
    (K_X+\D+(X_k)_{\re})^n\cdot F_i 
    &= (K_X+\D+(X_k)_{\re})^n\cdot F^{\nu}_i\\
    &= (K_X+\D+(X_k)_{\re}|_{F^{\nu}_i})^n\\
    &=(K_{F^{\nu}_i}+\Diff_{F^{\nu}_i}(\D+(X_k)_{\re}-F_i))^n\\
    &=\vol(K_{F^{\nu}_i}+\Diff_{F^{\nu}_i}(\D+(X_k)_{\re}-F_i))\\
    &\ge v(n,D(I))>0.
\end{aligned}
\]
This combined with (\ref{vge}) gives
$$\frac{v}{v(n,D(I))}\ge\sum_im_i\ge m_i.$$ Therefore, if
\[
p>\frac{v}{v(n,D(I))}
\]
then every component of $X_k$ has multiplicity $m_i$ less than $p$. In particular, every $m_i$ is prime to $p$. Thus the proof of [Theorem \ref{tame}, (2)$\implies$(1)] applies: let $L=\frac{K[t]}{\p-t^N}$, where $\p$ is a uniformizer of $R$ and $N$ is the least common multiple of all the multiplicities $m_i$. This induces a finite morphism $g$ from the normalization $X'$ of $X\times_RR_L$ to $X$ (where $R_L=\frac{R[t]}{\p-t^N}$ is a DVR). Then $(X',g^*\D)$ is the stable log model over $\Sp R_L$ with generic fibre $\simeq (X_K,\D_K)\times_KL$.
\end{proof}

\begin{remark}[Uniformity]\label{uniform}One can always take $N=(p-1)!$ in the proof of \ref{SRn1}.
\end{remark}

\begin{proof}[Proof of Corollary \ref{SR2lc}]
\hyperref[Vn]{Conjecture V\(_2\)} holds uniformly for any algebraically closed field $k$ by \cite[Theorem 2]{HK16}. On the other hand, by spreading out the morphism $(X_K,\D_K)\lr\Sp K$, we obtain a flat projective morphism $(X,\D)\lr\Sp R$ with generic fibre $(X_K,\D_K)$. By passing to a log resolution $\h X$ (\cite[Theorem 1.1]{CP19} and \cite[Corollary 1.5]{CJS20}), we may assume $(\h X,\h\D+(\h X_k)_{\re})$ is lc, where $\h\D$ denotes the strict transform of $\D$ plus the reduced exceptional divisors not contained in $\h X_k$. In this way, $(K_{\h X}+\h\D+(\h X_k)_{\re})|_{\h X_K}=K_{\h X_K}+\h\D_K$ has volume $=\vol(K_{X_K}+\D_K)>0$, so $K_{\h X}+\h\D+(\h X_k)_{\re}$ is big over $\Sp R$. Thus by \cite[Theorem 4.11]{HNT20}, the log canonical model $(X_1,\D_1+(X_{1,k})_{\re})$ over $\Sp R$ exists over any perfect field $k$ with $\ch k>5$. Its generic fibre $(X_{1,K},\D_{1,K})$ is the LCM of $(X_K,\D_K)$, and hence is isomorphic to $(X_K,\D_K)$, as $K_{X_K}+\D_K$ is lc and ample and the LCM is unique (\cite[3.52 (1)]{KM98}). Therefore, the conclusion of Theorem \ref{SRn1} holds when $n=2$ and $\ch k>\operatorname{max}\{5,\frac{v}{v(2,D(I))}\}$.
\end{proof}

Via Posva's gluing theorem \cite{Pos24} for threefolds in characteristic $p$, we obtain

\begin{corollary}[Stable reduction for slc canonically polarized surfaces]\label{SR2} 
Fix $v>0$ and a DCC set $I\subset(0,1]$ that contains $1$. If everything is defined over an algebraically closed field $k$ of characteristic $>\operatorname{max}\{5,\frac{v}{v(2,D(I))}\}$, with $\vol(K_{X_K}+\D_K)\le v$ and $\coeff(\D_K)\subseteq I$, then \hyperref[SRn]{Conjecture SR\(_2\)} holds more generally when $X_K$ is geometrically demi-normal, with $K_{X_K}+\D_K$ slc and ample over $K$.
\end{corollary}

\begin{proof}
Let $\coprod_i(X^i_K,D_K^i)\overset{\nu}{\lr}X_K$ be the normalization, with the induced involution $\ta_K$, where each $X^i_K$ is a connected component of the normalization of $X_K$ and $D_K^i$ is the conductor on $X_K^i$. Denote by $\D^i_K$ the divisorial part of $\nu^{-1}(\D_K)$ on $X^i_K$. Then, as $K_{X_K}+\D_K$ is slc and ample, $\nu^*(K_{X_K}+\D_K)=\sum_i(K_{X^i_K}+D_K^i+\D^i_K)$ is lc and ample. Moreover, for each $i$, $\vol(K_{X^i_K}+D_K^i+\D^i_K)\le\vol(K_{X_K}+\D_K)\le v$ and $\coeff(D_K^i+\D^i_K)=\coeff(\D^i_K)\subseteq\coeff(\D_K)\subseteq I$. Thus, by applying Corollary \ref{SR2lc} to each $(X^i_K,D_K^i+\D^i_K)$ (along with Remark \ref{uniform}), we find that if $\ch k>\operatorname{max}\{5,\frac{v}{v(2,D(I))}\}$, there exists a finite extension $L/K$ and pairs $(X'_i,D'_i+\D'_i)$ over $T':=\Sp R_L$ for each $i$, with generic fibre $\simeq(X_K^i,D_K^i+\D^i_K)\times_KL$, such that $(X'_i)_k$ is reduced, and $K_{X'_i}+D'_i+\D'_i+(X'_i)_k$ is lc and ample over $T'$. 

The rest follows from Step 2 of Posva's proof of \cite[Theorem 6.0.5]{Pos24}, where it is shown that we can "glue $(X'_i,D'_i+\D'_i)$ back" to obtain the stable log model. For ease of comparison with Posva's notation, we explain only the boundary-free case. The argument extends directly to the boundary case, as the Propositions and Lemmas used (specifically, \cite[9.55, 9.11]{Kol13}, \cite[9.4]{Pat17}, \cite[2.15, 2.16.3]{Kol23b}, and \cite[6.0.4, 3.4.1, 5.3.5]{Pos24}) are stated for, and hold in, the boundary case.

Thus we set $I=\{1\}$ and $\D_K=0$ above (and consequently, $\D^i_K=0$ and $\D'_i=0$). In particular, the bound on $\ch k$ for which Corollary \ref{SR2lc} holds becomes $\operatorname{max}\{5,\frac{v}{v(2,D(\{1\}))}\}$. We have now established, using Posva's notation therein, that 
$$K_{\Bar{X}_{\text{can}}}+\Bar{D}_{\text{can}}+(\Bar{X}_{\text{can}})_{t'}\,\,\text{is lc and ample, and}\,\,(\Bar{X}_{\text{can}})_L\simeq\coprod_i(X_K^i,D_K^i)\times_KL,$$ where $t'\in T'$ denotes the closed point, and
$$\Bar{X}_{\text{can}}=\coprod_iX'_i,\quad\Bar{D}_{\text{can}}=\coprod_iD'_i,\quad\text{and}\quad(\Bar{X}_{\text{can}})_{t'}=\coprod_i(X'_i)_k$$
in our notation. It is then shown that the pullback of the involution $\ta_K$ extends to an involution $\ta_{\text{can}}$ on the normalization of $\Bar{D}_{\text{can}}$ satisfying desired properties, and thus there exists a geometric quotient $\nu\colon\Bar{X}_{\text{can}}=\coprod_iX'_i\longrightarrow X_{\text{can}}$ over $T'$. As geometric quotient commutes with flat base change \cite[9.11]{Kol13}, $X_{\text{can}}\longrightarrow T'$ has generic fibre isomorphic to the quotient of $\coprod_i(X_K^i,D_K^i)\times_KL$, namely $X_K\times_KL$. Moreover, $X_{\text{can}}$ is demi-normal and $K_{X_{\text{can}}}$ is $\Q$-Cartier by \cite[Propositions 3.4.1 and 5.3.5]{Pos24}, respectively. Finally, since $\nu^*(K_{X_{\text{can}}}+(X_{\text{can}})_{t'})=K_{\Bar{X}_{\text{can}}}+\Bar{D}_{\text{can}}+(\Bar{X}_{\text{can}})_{t'}$ is lc and ample, $K_{X_{\text{can}}}+(X_{\text{can}})_{t'}$ is slc and ample. Thus $(X_{\text{can}},0)$ is the stable (log) model over $T'$. With the presence of $\D_K$, the above argument produces a stable log model $(X_{\text{can}},\D_{\text{can}})$ over $T'$ with generic fibre $\simeq(X_K,\D_K)\times_KL$.
\end{proof}

\begin{proof}[Proof of Theorem \ref{proper}]
Set $I=\{1\}$ and $\D_K=0$ in Corollary \ref{SR2}. We will show that the valuative criterion of properness holds for families of stable surfaces of volume $v\le v_0$ when $\ch k>\operatorname{max}\{5,\frac{v_0}{v(2,D(\{1\}))}\}$. Let $(X_K\lr\Sp K)\in\overline{\mathscr{M}}_{2,v,k}(\Sp K)$, where $K$ is the fraction field of a DVR $R$ with algebraically closed residue field $k$ of characteristic $>\operatorname{max}\{5,\frac{v_0}{v(2,D(\{1\}))}\}$. By Corollary \ref{SR2}, there exists a finite field extension $L/K$ and a stable (log) model $X_{\text{can}}$ over $\Sp R_L$ with generic fibre $\simeq X_K\times_KL$. It then remains, by the proof of \cite[Theorem 6.0.5]{Pos24}, to show that condition $(S2)$ stated therein holds, namely by proving that the central fibre $(X_{\text{can}})_k$ is $S_2$. This was established in \cite[Theorem 4.12]{ABP23} for $\ch k>5$. Therefore, we conclude that $(X_{\text{can}}\lr \Sp R_L)\in\overline{\mathscr{M}}_{2,v,k}(\Sp R_L)$. This completes the valuative criterion of properness, and thus $\overline{\mathscr{M}}_{2,v,k}$ is proper when $v\le v_0$ and $\ch k>\operatorname{max}\{5,\frac{v_0}{v(2,D(\{1\}))}\}$. The projectivity of $\overline{\text{M}}_{2,v,k}$ then follows as in the proof of \cite[Theorem 4.13]{ABP23}, namely from \cite[Theorem 1.2 (2)]{Pat17}. 
\end{proof}

When $k=\C$, we show that the log canonical model of $(X,\D+(X_k)_{\re})$ over $\Sp R$ exists under mild assumptions:

\begin{corollary}\label{lcm0}
Suppose everything here is defined over $k=\C$. Let $f\colon X\longrightarrow \Sp R$ be a projective morphism from a quasi-projective normal variety $X$ onto the spectrum of a DVR $R$ with residue field $k$. Assume that $(X,\D+(X_k)_{\re})$ is lc, and that $(X,\D)$ is klt away from the central fibre $X_k$. If $K_X+\D$ is $f$-big, then the log canonical model of $(X,\D+(X_k)_{\re})$ over $\Sp R$ exists.
\end{corollary}
\begin{proof}
Let $N$ be the least common multiple of all the multiplicities of the components of $X_k$. Let $R'=\frac{R[t]}{\p-t^N}$, where $\p$ is a uniformizer of $R$, and make the base change $T'=\Sp R'\lr T=\Sp R$. Then the normalization $X'$ of $X\times_TT'$ has central fibre $X'_k$ reduced (Lemma \ref{bcq}). Let $g\colon X'\lr X$ and $f'\colon X'\lr T'$ denote the induced morphisms, and set $\D':=g^*\D$. We first show that $K_{X'}+\D'$ is $\Q$-Cartier, $f'$-big, and klt. By Lemma \ref{bc} we have
\[
g^*(K_X+\D+(X_k)_{\re})\sim_{\Q}K_{X'}+\D'+X'_k.
\]
As $K_X+\D+(X_k)_{\re}$ and $X'_k$ are $\Q$-Cartier, so is $K_{X'}+\D'$. Since pullback commutes with restriction, we have $$g^*(K_{X_K}+\D_K)=g^*((K_X+(X_k)_{\re}+\D)|_{X_K})\sim_{\Q}(K_{X'}+X'_k+\D')|_{X'_L}=K_{X'_L}+\D'_L$$
where $K,L$ denote the fraction field of $R,R'$ respectively. As $K_{X_K}+\D_K$ is big, so is $K_{X'_L}+\D'_L$, i.e. $K_{X'}+\D'$ is $f'$-big. 

To show $K_{X'}+\D'$ is klt, let us show first that $K_X+\D+(1-\frac{1}{N})(X_k)_{\re}$ is klt. Let $E$ be a divisor over $X$. By \cite[2.27]{KM98} and the fact that $(X_k)_{\re}$ is $\Q$-Cartier (as $(X,\D+(X_k)_{\re})$ and $(X,\D)$ are pairs), if $\operatorname{center}_XE\subset\Supp X_k$, then the discrepancy satisfies $a(E,X,\D+(1-\frac{1}{N})(X_k)_{\re})>a(E,X,\D+(X_k)_{\re})$; if $\operatorname{center}_XE\not\subset\Supp X_k$, then $a(E,X,\D+(1-\frac{1}{N})(X_k)_{\re})=a(E,X,\D)$. Therefore, as $(X,\D+(X_k)_{\re})$ is lc and $(X,\D)$ is klt away from $X_k$, $K_X+\D+(1-\frac{1}{N})(X_k)_{\re}$ is klt. Now write $X_k=\sum_im_iF_i$, where $F_i$ denotes a component of $X_k$ with multiplicity $m_i$, and let $F'_{ij}$ be the preimages of $F_i$ under $g$, with $e_{ij}$ their ramification indices. Then we have $g^*F_i=\sum_je_{ij}F'_{ij}$ and $K_{X'}\sim_{\Q}g^*K_X+\sum_{i,j}(e_{ij}-1)F'_{ij}$. Therefore, as in \cite[2.41.5]{Kol13}, we obtain
\[
\begin{aligned}
&\quad\,\, g^*(K_X+\D+(1-\frac{1}{N})(X_k)_{\re})\\
&\sim_{\Q} K_{X'}-\sum_{i,j}(e_{ij}-1)F'_{ij}+\D'+(1-\frac{1}{N})\sum_{i,j}e_{ij}F'_{ij}\\
&= K_{X'}+\D'+\sum_{i,j}(1-\frac{e_{ij}}{N})F'_{ij}.
\end{aligned}
\]
Since $K_X+\D+(1-\frac{1}{N})(X_k)_{\re}$ is klt, so is its pullback $K_{X'}+\D'+\sum_{i,j}(1-\frac{e_{ij}}{N})F'_{ij}$ by Proposition \ref{kol}. Since $\sum_{i,j}(1-\frac{e_{ij}}{N})F'_{ij}$ is effective (as $e_{ij}\le N$), we conclude that $K_{X'}+\D'$ is klt.

Now, as $K_{X'}+\D'$ is klt, $\Q$-Cartier, and $f'$-big, by \cite[Theorem 1.2 (3)]{BCHM10} the $\oo_{T'}$-algebra
$$
\mathfrak{R}(f',\D')=\bigoplus_{m\ge0} f'_*\oo_{X'}(\lfloor m(K_{X'}+\D') \rfloor)
$$
is finitely generated, whose Proj over $T'$ is the log canonical model of $(X',\D')$ over $T'$. But this algebra is the same as the algebra
$$
\mathfrak{R}(f',\D'+X'_k)=\bigoplus_{m\ge0} f'_*\oo_{X'}(\lfloor m(K_{X'}+\D'+X'_k) \rfloor),
$$
as $X'_k$ is linearly $f'$-trivial. Thus the log canonical model of $(X',\D'+X'_k)$ over $T'$ is the same as that of $(X',\D')$, which exists and will be denoted by $(X'_1,\D'_1+X'_{1,k})$. The rest then follows as in the proof of [Theorem \ref{tame}, (1)$\implies$(2)]: the Galois group $G$ associated to the base change $g$ acts on $X'_1$, thus we can take the geometric quotient $g_1\colon X'_1\lr X'_1/G=:X_1$ and obtain
$$
g_1^*(K_{X_1}+\D_1+(X_{1,k})_{\re})=K_{X'_1}+\D'_1+X'_{1,k}
$$
where $\D_1=\frac{g_{1*}\D'_1}{\deg g_1}$. As $K_{X'_1}+\D'_1+X'_{1,k}$ is lc and ample over $T'$, so is $K_{X_1}+\D_1+(X_{1,k})_{\re}$ over $T$. Thus $(X_1,\D_1+(X_{1,k})_{\re})$ is the log canonical model over $T$.
\end{proof}

\begin{corollary}\label{SRn0}
\hyperref[SRn]{Conjecture SR\(_n\)} holds over $k=\C$, provided that $(X_K,\D_K)$ is klt. 
\end{corollary}
\begin{proof}
By spreading out $(X_K,\D_K)\lr\Sp K$, we obtain a flat projective morphism $(X,\D)\lr\Sp R$ with generic fibre $(X_K,\D_K)$. By Hironaka, there exists a log resolution $\h X$ of $(X,\D+(X_k)_{\re})$ such that the sum of the strict transform of $\D+(X_k)_{\re}$ and $\sum_iE_i$ has simple normal crossing support, where $E_i$ denote the exceptional divisors. Set $\h\D:=(\text{strict transform of\,\,}\D)-\sum_i a(E_i,X,\D)E_i$, where the sum runs over those $E_i\not\subset\Supp\h X_k$ and with $a(E_i,X,\D)<0$. In this way, $(K_{\h X}+\h\D)|_{\h X_K}=K_{\h X_K}+\h\D_K$ has volume $=\vol(K_{X_K}+\D_K)>0$, so $K_{\h X}+\h\D$ is big over $\Sp R$. Furthermore, as $(X_K,\D_K)$ is klt, $\coeff\h\D\subset(0,1)$ by the construction of $\h\D$, and thus $(\h X,\h\D)$ is klt by \cite[2.31 (3)]{KM98}. Therefore by Corollary \ref{lcm0} the log canonical model $(X_1,\D_1+(X_{1,k})_{\re})$ over $\Sp R$ exists. Its generic fibre $(X_{1,K},\D_{1,K})$ is the LCM of $(X_K,\D_K)$, and hence is isomorphic to $(X_K,\D_K)$, as $K_{X_K}+\D_K$ is lc and ample and the LCM is unique. Hence we are done by the proof of [Theorem \ref{tame}, (2)$\implies$(1)].
\end{proof}

As mentioned in the introduction, when $\ch k=0$ or sufficiently large, we can obtain the stable model of a family $X\lr T$ of varieties of general type by first passing to its LCM $(X,(X_k)_{\re})$ over $T$, and then making a base change. In particular, the stable limit is the preimage of the central fibre $X_k$ of the LCM under the base change, where we describe its components in the following 

\begin{corollary}[Stable limits]\label{SL} Suppose everything here is defined over an algebraically closed field $k$. Let $R$ be a DVR with fraction field $K$ and residue field $k$. Let $X_K$ be a projective, geometrically normal, and geometrically connected variety over $K$, and let $\D_K$ be an effective divisor on $X_K$ such that $K_{X_K}+\D_K$ is log canonical and ample. Assume that there exists an LCM $(X,\D+(X_k)_{\re})$ over $\Sp R$ with generic fibre $\simeq(X_K,\D_K)$, and that either 

$(1)$ $\ch k=0$, or 

$(2)$ $\ch k=p$, every component of the central fibre $X_k$ of the LCM has multiplicity prime to $p$, and a log resolution $\h X$ of $(X,\D+(X_k)_{\re})$ exists. 

Then, the stable log model $(X',\D')$ exists after a tamely ramified base change $g$, and the stable limit $X'_k$ is $g^{-1}(X_k)$, with its components described as follows. Let $F$ be a component of $X_k$ with multiplicity $m_F$. Then $g^{-1}(F)$ consists of $\#_F:=\operatorname{gcd}\{m_F\,,\,n_j\}_j$ reduced components $F'_1,\cdots, F'_{\#_F}$ in the stable limit, each of which is a degree $\frac{m_F}{\#_F}$ cover of $F$, where $n_j$ are the multiplicities of the components in the log resolution $\h X$ that meet the strict transform of $F$. Moreover, each $F'_i$ has $$\vol(K_{F'_i}+B'_i)=\frac{m_F}{\#_F}\cdot\vol(K_{F}+B),$$ where $B'_i=\Diff_{F'_i}(\D'+X'_k-F'_i)$ and $B=\Diff_F(\D+(X_k)_{\re}-F)$. 
\end{corollary}

\begin{proof}
We treat only the $\ch k=p$ case, which will cover the $\ch k=0$ case. Start with the LCM $(X,\D+(X_k)_{\re})$ over $\Sp R$. By Theorem \ref{tame}, we obtain the stable log model $(X',\D')$ after a tame base change $g\colon \Sp R'\lr \Sp R$, where $X'$ is the normalization of $X\times_RR'$ and $\D'=g^*\D$. Denoting again by $g\colon X'\lr X$ the induced morphism, we have $$K_{X'}+\D'+X'_k\sim_{\Q}g^*(K_X+\D+(X_k)_{\re}).$$ Since pullback commutes with restriction, for each component $F$ of $X_k$ and any preimage $F'_i\subseteq g^{-1}(F)$, we have $$(K_{X'}+\D'+X'_k)|_{F'_i}\sim_{\Q}(g|_{F'_i})^*((K_X+\D+(X_k)_{\re})|_F).$$ Thus by adjunction we obtain
\begin{equation}\label{g|*}
K_{F'_i}+B'_i\sim_{\Q}(g|_{F'_i})^*(K_F+B),
\end{equation}
where $B'_i=\Diff_{F'_i}(\D'+X'_k-F'_i)$ and $B=\Diff_F(\D+(X_k)_{\text{red}}-F)$. Since $R$ contains $k=\ol k$, $g$ is Galois. Thus by Lemma \ref{ef}, we have $\deg g=\#_F\cdot e\cdot\deg(g|_{F'_i})$; while by Lemma \ref{val}, $\deg g=m_F\cdot e$. Hence $\deg (g|_{F'_i})=\frac{m_F}{\#_F}$. Thus (\ref{g|*}) implies
\[
\vol(K_{F'_i}+B'_i)=\frac{m_F}{\#_F}\cdot\vol(K_F+B).
\]

It remains to show that $g^{-1}(F)$ consists of $\#_F=\text{gcd}\{m_F\,,\,n_j\}_j$ components in the stable limit $X'_k$. Note first that $\#_F$ is the same as the number of components of $g^{-1}(\h F)$, where $\h F$ is the strict transform of $F$ on any log resolution $\h X$. Indeed, by Lemma \ref{ef}, $\deg g=\#_F\cdot e\cdot f$, where $\deg g,e,$ and $f$, and hence $\#_F$, are all independent of the model to which the lemma is applied. Therefore, we pass to a log resolution $\widehat{X}$ of $(X,\D+(X_k)_{\re})$ and compute the number of components of $g^{-1}(\h F)$. Locally formally near a point $\h Q\in \h F$ in $\h X$, namely in the completion $\h\oo_{\h X,\h Q}$, we have
\begin{equation}\label{tq}
    \p=(\text{unit})x_0^{m_F}x_1^{n_1}\cdots x_s^{n_s}
\end{equation}
where $\p$ is a uniformizer of the DVR $R$, $x_0$ defines $\h F$ with multiplicity $m_F$, and each of the remaining $x_j$ defines a component with multiplicity $n_j$ that contains $\h Q$. By assumption $m_F$ is prime to $p$, so the unit has an $m_F$-th root ($\h\oo_{\h X,\h Q}$ is Henselian with residue field $k=\ol k$) and hence can be absorbed by $x_0^{m_F}$. Let $q$ be a prime that divides $m_F$ and make a base change $\p=t^q$. Then equation (\ref{tq}) becomes
\begin{equation}\label{x0xs}
    t^q=x_0^{m_F}x_1^{n_1}\cdots x_s^{n_s}.
\end{equation}
Suppose $q$ divides $n_1,\cdots, n_{\alpha-1}$. Then the normalizaton of (\ref{x0xs}) introduces 
\[
v=\frac{t}{x_0^{\frac{m_F}{q}}x_1^{\frac{n_1}{q}}\cdots x_{\alpha-1}^{\frac{n_{\alpha-1}}{q}}x_{\alpha}^{\lfloor\frac{n_{\alpha}}{q}\rfloor}\cdots x_s^{\lfloor\frac{n_s}{q}\rfloor}}
\]
and the equation becomes
\begin{equation}\label{v^q}
v^q=x_{\alpha}^{n'_{\alpha}}\cdots x_s^{n'_s}
\end{equation}
where $1\le n'_{j}\le q-1$. Since $v^q-x_{\alpha}^{n'_{\alpha}}\cdots x_s^{n'_s}=\prod_{i=1}^q(v-\xi^i_q x_{\alpha}^{\frac{n'_{\alpha}}{q}}\cdots x_s^{\frac{n'_s}{q}})$, where $\xi_q$ is a primitive $q$-th root of unity, and none of the factors $(v-\xi^i_q x_{\alpha}^{\frac{n'_{\alpha}}{q}}\cdots x_s^{\frac{n'_s}{q}})$ lives in the ring to which $v^q-x_{\alpha}^{n'_{\alpha}}\cdots x_s^{n'_s}$ belong, we see that equation (\ref{v^q}) factors into $q$ factors if and only if $x_{\alpha}^{n'_{\alpha}}\cdots x_s^{n'_s}=1$, meaning that $q$ divides all $n_1,\cdots ,n_s$. Therefore, under the base change $\p=t^q$, $F$ splits into $q$ components if and only if every point $\h Q\in\h F$ has $q$ preimages, which is equivalent to the equation (\ref{v^q}) factoring into $q$ factors for every $\h Q\in \h F$. This happens precisely when every component that meets $\h F$ in the log resolution has multiplicity $n_j$ divisible by $q$. Repeating this procedure for each prime $q$ that divides $\text{gcd}\{m_F\,,\,n_j\}_j$, we find that $g^{-1}(F)$ consists of $\text{gcd}\{m_F\,,\,n_j\}_j$ components in the stable limit.
\end{proof}

\begin{remark}
Corollary \ref{SL} holds unconditionally when $k=\C$: the stable log model $(X',\D')$ over $\Sp R'$ exists by \cite[1.5]{HX13}. Taking the quotient of $X'$ by $G$ as in the last paragraph of Corollary \ref{lcm0} produces the LCM $(X,\D+(X_k)_{\re})$ over $\Sp R$ with generic fibre $\simeq(X_K,\D_K)$.
\end{remark}

\appendix

\section{Lemmas}

This section collects some lemmas used in the proofs, which are provided here for convenience. Throughout, let $X\longrightarrow\Sp R$ be a morphism from a Noetherian normal scheme $X$ onto the spectrum of a DVR $R$ with algebraically closed residue field $k$. A base change $g\colon\Sp R'\longrightarrow\Sp R$ induces a morphism (still denoted by $g$) from the normalization $X'$ of $X\times_RR'$ to $X$. 

\begin{lemma}\label{val}
Let $g\colon X'\longrightarrow X$ be the morphism induced from a totally ramified base change (such as $\p=t^N$). Let $D$ be a component of the central fibre of $X$, and $D'$ a component of $g^{-1}(D)$, with $m_D,m_{D'}$ their multiplicities respectively. Then
\[
m_D\cdot e_{D'}=\deg g\cdot m_{D'},
\]
where $e_{D'}$ denotes the ramification index of $D'$ with respect to $g$.
\end{lemma}
\begin{proof}
The valuation of $D'$ can be computed by pulling back a uniformizer $\p$ of the base $\Sp R$ in two ways, which agree: first pullback to $X$ giving $m_D$, then pullback by $g$ giving $m_D\cdot e_{D'}$; or first pullback by $g\colon\Sp R'\longrightarrow\Sp R$ giving $e_{0'}=\deg g$ (as $g$ is totally ramified), then pullback to $X'$ giving $\deg g\cdot m_{D'}$. 
\end{proof}

\begin{lemma}\label{ef} Let $g\colon X'\longrightarrow X$ be the morphism induced from a totally ramified, Galois base change. Let $D$ be a component of the central fibre of $X$. Then 
\[
\deg g=\#\cdot e\cdot f
\]
where $\#$ is the number of components $D'_i$ of $g^{-1}(D)$, $e$ is the ramification index of $D'_i$, and $f$ is the degree $\deg(g|_{D'_i}\colon D'_i\longrightarrow D)$.
\end{lemma}
\begin{proof}
We may assume $X$ is connected. By \cite[Proposition 7.1.38]{Liu02} we have
\[
\deg g=\sum_i e_{D'_i}\cdot f_{D'_i}.
\]
As the base change is Galois, $e_{D'_i}=e$ and $f_{D'_i}=f$ are independent of $i$.
\end{proof}

\begin{lemma}\label{bcq}
Let $g\colon X'\lr X$ be the morphism induced from the base change $\p=t^q$, where $q$ is a prime $\ne\ch k$. If a component $D$ of the central fibre of $X$ has multiplicity $m$ that is divisible by $q$, then each component of $g^{-1}(D)$ has multiplicity $\frac{m}{q}$. In particular, $g^{-1}(D)$ is reduced under the base change $\p=t^N$ if $m|N$ and $\ch k\nmid N$.
\end{lemma}
\begin{proof}
Since $X$ is normal, locally formally $D\subset X$ can be described by $\p=ux^m$, where $u$ is a unit in $\h\oo_{X,D}$. Since $q\ne\ch k$, the unit $u$ has a $q$-th root $u'$ in $\h\oo_{X,D}$ by Hensel's lemma, so the equation becomes $\p=(u'x^{\frac{m}{q}})^q$. Making the base change $\p=t^q$, the equation becomes $t^q=(u'x^{\frac{m}{q}})^q$. Its normalization introduces $v=\frac{t}{u'x^{\frac{m}{q}}}$, so the equation then becomes $v^q=1$. Since $k$ is algebraically closed and $q\ne\ch k$, $v^q=1$ has $q$ factors, meaning locally formally $D$ has $q$ preimages. By Lemma \ref{ef} we have $q=\#\cdot e\cdot f$. Since $q$ is a prime, $D$ having $q$ preimages locally means that either $\#=q$ so $D$ has $q$ preimages $D'_1,\cdots,D'_q$ and $e=f=1$, or $\#=1$ meaning that $D$ has one preimage $D'$ that maps $q$ to $1$ onto $D$, $f=q$ and $e=1$. Either case $e=1$. By lemma \ref{val}, we have $m\cdot 1=q\cdot m_{D'_i}$ or $m\cdot 1=q\cdot m_{D'}$. Either case shows that each component of $g^{-1}(D)$ has multiplicity $\frac{m}{q}$.
\end{proof}

\begin{lemma}\cite[2.41.4]{Kol13}\label{bc} 
Let $g\colon X'\longrightarrow X$ be a tamely ramified cover of demi-normal schemes. Let $F$ be a divisor on $X$ whose support contains the branch divisor of $g$, and let $\D$ be a $\Q$-divisor on $X$. Then $F$ is reduced if and only if the following canonical $\Q$-linear equivalence holds: $$K_{X'}+g^*\D+(g^*F)_{\re}\sim_{\Q}g^*(K_X+\D+F).$$
\end{lemma}
\begin{proof}
By the Hurwitz formula \cite[2.41]{Kol13} we have $K_{X'}\sim g^*K_X+\mathcal{R}$, where $\mathcal{R}$ is the ramification divisor, and thus $$K_{X'}+g^*\D+g^*F-\mathcal{R}\sim_{\Q}g^*(K_X+\D+F).$$ We show that $g^*F-\mathcal{R}\sim_{\Q}(g^*F)_{\re}$ iff $F$ is reduced. Write $F=\sum_ia_iF_i$ and $g^*F_i=\sum_je_{ij}F'_{ij}$, where $F'_{ij}$ are the preimages of $F_i$ with ramification indices $e_{ij}$. Then $g^*F-\mathcal{R}=\sum_{i,j}a_ie_{ij}F'_{ij}-\sum_{i,j}(e_{ij}-1)F'_{ij}$ and $(g^*F)_{\re}=\sum_{i,j}F'_{ij}$. Since all $\Q$-linear equivalences here are canonical, we have $g^*F-\mathcal{R}\sim_{\Q}(g^*F)_{\re}\iff(a_i-1)e_{ij}+1=1\iff a_i=1\iff F$ is reduced.
\end{proof}

\begin{lemma}\label{cyclic}
Let $R$ be a Henselian DVR with algebraically closed residue field $k$ and fraction field $K$. If $\ch k\nmid N$, then any field extension $L/K$ of degree $N$ is Galois with cyclic group $G$ of order $N$.
\end{lemma}
\begin{proof}
Since $R$ is Henselian, the integral closure $R_L$ of $R$ in $L$ is a Henselian DVR. By \cite[\href{https://stacks.math.columbia.edu/tag/09E8}{Tag 09E8}]{stacks-project}, as $L/K$ is separable, $R_L$ is local, and $k=\ol k$, we have 
$$N=[L:K]=\sum_ie_if_i=ef=e.$$
Thus $L/K$ is totally ramified. If $\p,t$ are uniformizers of $R$ and $R_L$ respectively, then $\p=ut^N$, where $u$ is a unit in $R_L$. By Hensel's lemma, since $\ch k\nmid N$ and $k=\ol k$, the unit $u$ has an $N$-th root $u^{\frac{1}{N}}$ in $R_L$. Replacing $t$ with $u^{\frac{1}{N}}t$, $\p=ut^N$ becomes $\p=t^N$, i.e. $L=\frac{K[t]}{\p-t^N}$. Since $R_L$ contains a primitive $N$-th root of unity $\xi_N$, $L/K$ is Galois, and the automorphism $\sigma$ of $L$ that fixes $K$ and sends $t$ to $\xi_N t$ generates the Galois group $G$. Hence $G=\langle\sigma\rangle$ is cyclic of order $N$.
\end{proof}

\begin{lemma}\cite[Lemma 4.4]{MP04}\label{MP}
Let $I\subseteq (0,1]$ be any subset that contains $1$. Then $I$ is a DCC set if and only if $D(I)$ is a DCC set, where $$D(I)=\{\,a=\frac{m-1+f}{m}\,|\, m\in\mathbb{N},\, f\in I_+\cap[0,1]\,\}$$ 
and $I_+$ denotes the union of $\{0\}$ and the set of all finite sums of elements in $I$.
\end{lemma}

\bibliographystyle{amsalpha}

\newcommand{\etalchar}[1]{$^{#1}$}
\providecommand{\bysame}{\leavevmode\hbox to3em{\hrulefill}\thinspace}
\providecommand{\MR}{\relax\ifhmode\unskip\space\fi MR }
\providecommand{\MRhref}[2]{
  \href{http://www.ams.org/mathscinet-getitem?mr=#1}{#2}
}

\end{document}